\documentclass[12pt,a4paper]{amsart}

\usepackage{enumerate}
\usepackage{amsmath,amsthm,verbatim,amssymb,amsfonts,amscd,graphicx}
\usepackage{tikz-cd} 
\usepackage{graphics}
\usepackage{hyperref}
\usepackage{mathrsfs}
\usepackage{faktor}
\usepackage[margin=2.9cm]{geometry}
\usepackage[latin2]{inputenc}
\theoremstyle{plain}
\newtheorem{theorem}{Theorem}
\newtheorem*{theorem*}{Theorem}
\newtheorem{corollary}{Corollary}
\newtheorem{lemma}{Lemma}

\theoremstyle{definition}

\begin{document}
\title{An inequality in mixed multiplicities}
\author{Suprajo Das}
\address{Suprajo Das, Chennai Mathematical Institute, Kelambakkam 603103, India}
\email{dassuprajo@gmail.com}
\begin{abstract}
    The theory of mixed multiplicities of (not necessarily Noetherian) filtrations of $m_R$-primary ideals in a Noetherian local ring $R$, has been developed by Cutkosky, Sarkar and Srinivasan. The objective of this article is to generalise a Minkowski type inequality given in their paper. We also recover a result of Cutkosky, Srinivasan and Verma as a simple consequence of our inequality.
\end{abstract}
\maketitle
\section{introduction}
The concept of mixed multiplicities of ideals is of remarkable importance in the areas of commutative algebra and algebraic geometry. A survey of the theory of mixed multiplicities of ideals, with proofs, can be found in \cite{TV} and chapter $17$ of \cite{Hun}. Recently, the study of mixed multiplicities of (not necessarily Noetherian) filtrations of $m_R$-primary ideals in a Noetherian local ring $R$, has been pioneered by Cutkosky, Sarkar and Srinivasan in \cite{Cut}. This generalises the classical theory of mixed multiplicities of ideals. The interested reader is requested to refer to \cite{Cut} for notations and definitions. Among many other things, the authors prove the following result.
\begin{theorem*}[Theorem $6.3.$ \cite{Cut}]
Suppose that $(R,m_R)$ is a Noetherian local ring of dimension $d\geq 1$ such that $\dim N(\hat{R})<d$, $M$ is a finitely generated $R$-module and $\mathcal{I}(1) = \{I(1)_j\}$ and $\mathcal{I}(2) = \{I(2)_j\}$ are filtrations of $R$ by $m_R$-primary ideals. Then for all $0\leq i\leq d$, $$e\left(\mathcal{I}(1)^{[d-i]},\mathcal{I}(2)^{[i]};M\right)^d \leq e\left(\mathcal{I}(1);M\right)^{d-i}e\left(\mathcal{I}(2);M\right)^i.$$
\end{theorem*}
In this manuscript we generalize the above inequality. More specifically we prove the following theorem.
\begin{theorem*}[Theorem \ref{mainthm}]
Suppose that $(R,m_R)$ is a Noetherian local ring of dimension $d\geq 1$ such that $\dim N(\hat{R})<d$, $M$ is a finitely generated $R$-module and $\mathcal{I}(1) = \{I(1)_j\}, \ldots, \mathcal{I}(r) = \{I(r)_j\}$ are filtrations of $R$ by $m_R$-primary ideals. Then $$e\left(\mathcal{I}(1)^{[d_1]},\ldots,\mathcal{I}(r)^{[d_r]};M\right)^d \leq \prod\limits_{i=1}^r e\left(\mathcal{I}(i);M\right)^{d_i}$$ for any $d_1,\ldots,d_r \in\mathbb{N}$ satisfying $d_1+\cdots+d_r=d$. 
\end{theorem*}
We also provide an application of the above inequality.
\section{Minskowski type inequalities for mixed multiplicities}
We adhere to \cite{Hun} for their notations and definitions regarding mixed multiplicities of ideals. The following theorem appears in Corollary $17.7.3$ of \cite{Hun}.
\begin{theorem}\label{ineq1}
Let $(R,m_R)$ be a Noetherian local ring of dimension $d\geq 1$ and $I$ and $J$ $m_R$-primary ideals. Then for all $0\leq i\leq d$,
$$e\left(I^{[d-i]},J^{[i]}; R\right)^d \leq e(I;R)^{d-i}e(J;R)^i.$$
\end{theorem}
The following theorem is a result of Teissier (page $306$ of \cite{BT}). A detailed proof can be found in Theorem $17.4.6.$ and Proposition $17.2.2$ of \cite{Hun}.
\begin{theorem}\label{superficial}
Let $(R,m_R)$ be a Noetherian local ring of dimension $d=\dim R \geq 2$ having infinite residue field and $I_1,\ldots,I_r$ $m$-primary ideals in $R$. There exists an element $f\in I_1\setminus \left(\bigcup\limits_{P\in\mathrm{Min}(R)}P\right)$ such that if we denote the quotient ring $\faktor{R}{(f)}$ by $\Bar{R}$ and the images of the ideals $I_i$ in $\Bar{R}$ by $\Bar{I}_i$ respectively, then for any choice of non-negative integers $d_1,\ldots,d_r$ with $d_1+\cdots+d_r = d$ and $d_1\geq 1$, gives $$e\left(I_1^{[d_1]},\ldots,I_r^{[d_r]};R\right) = e\left(\Bar{I}_1^{[d_1-1]},\Bar{I}_2^{[d_2]},\ldots,\Bar{I}_r^{[d_r]};\Bar{R}\right).$$
\end{theorem}
The next theorem appears in Theorem $1.6.7$ of \cite{Laz} and generalises Theorem \ref{ineq1}. We reproduce their proof here.
\begin{theorem}\label{ineq2}
Let $(R,m_R)$ be a Noetherian local ring of dimension $d=\dim R \geq 1$ and $I_1,\ldots,I_d$ $m_R$-primary ideals in $R$. Then $$e\left(I_1,\ldots,I_d; R\right)^d \leq \prod\limits_{i=1}^d e(I_i; R).$$
\end{theorem}
\begin{proof}
We proceed by induction on $d$. The theorem is trivial when $d=1$, hence assume that $d\geq 2$. As mixed multiplicities are preserved under passage to the faithfully flat extension $R[X]_{m_RR[X]}$ where $X$ is a variable over $R$, so without loss of generality we may assume that the residue field of $R$ is infinite. We know from Theorem \ref{superficial} that there exists an element $f\in I_d$ such that if we denote the quotient ring $\faktor{R}{(f)}$ by $\Bar{R}$ and the images of the ideals $I_i$ in $\Bar{R}$ by $\Bar{I}_i$ respectively, then the following holds:
\begin{align}
    e\left(I_1,\ldots,I_d;R\right) &= e\left(\Bar{I}_1,\ldots,\Bar{I}_{d-1};\Bar{R}\right)\label{Eq1}\\
    e\left(I_i^{[d-1]},I_d;R\right) &= e\left(\Bar{I}_i;\Bar{R}\right) &&(\text{for all $1\leq i\leq d-1$}) \label{Eq2}
\end{align}
Now we observe that
\begin{align*}
    e\left(I_1,\ldots,I_d;R\right)^{d-1} &= e\left(\Bar{I}_1,\ldots,\Bar{I}_{d-1};\Bar{R}\right)^{d-1} && (\text{from $(\ref{Eq1})$})\\
    &\leq \prod\limits_{i=1}^{d-1} e\left(\Bar{I}_i;\Bar{R}\right) && (\text{by induction step})\\
    &= \prod\limits_{i=1}^{d-1} e\left(I_i^{[d-1]},I_d;R\right) && (\text{from $(\ref{Eq2})$})\\
    &\leq \prod\limits_{i=1}^{d-1} \left(e\left(I_i;R\right)^{\frac{d-1}{d}}e\left(I_d;R\right)^{\frac{1}{d}}\right) && (\text{from Theorem \ref{ineq1}})\\
    &= \prod\limits_{i=1}^{d} e\left(I_i;R\right)^{\frac{d-1}{d}}.
\end{align*}
So far we have gotten that $$e\left(I_1,\ldots,I_d;R\right)^{d-1} \leq \prod\limits_{i=1}^{d} e\left(I_i;R\right)^{\frac{d-1}{d}}$$ and the desired inequality follows by raising both sides to the exponent $\frac{d}{d-1}$.
\end{proof}
\begin{corollary}\label{ineq3}
Let $(R,m_R)$ be a Noetherian local ring of dimension $d=\dim R \geq 1$ and $I_1,\ldots,I_r$ $m_R$-primary ideals in $R$. Then $$e\left(I_1^{[d_1]},\ldots,I_r^{[d_r]}; R\right)^d \leq \prod\limits_{i=1}^r e(I_i; R)^{d_i}$$ where $d_1,\ldots,d_r\in\mathbb{N}$ such that $d_1+\cdots+d_r = d$.
\end{corollary}
\begin{proof}
First observe that $$e\left(I_1^{[d_1]},\ldots,I_r^{[d_r]}; R\right) = e\left(\overbrace{I_1,\ldots,I_1}^{d_1},\ldots,\overbrace{I_r,\ldots,I_r}^{d_r};R\right)$$ where each $I_i$ is repeated $d_i$-times. Now our assertion is a direct consequence of Theorem \ref{ineq2} and the above observation.
\end{proof}
We shall also require a combinatorial lemma due to Alon and Tarsi.
\begin{lemma}[Lemma $2.1.$ \cite{Alon}]\label{null}
Let $H = H(x_1,\ldots,x_n)$ be a polynomial in $n$ variables over an arbitrary field $F$. Suppose that the degree of $H$ as a polynomial in $x_i$ is at most $t_i$, for $1\leq i\leq n$, and let $S_i \subset F$ be a set of at least $t_i + 1$ distinct members of $F$. If $H(x_1,\ldots,x_n) = 0$ for all $n$-tuples $\left(x_1,\ldots,x_n\right)\in S_1\times \cdots \times S_n$, then $H\equiv 0$.
\end{lemma}
We borrow the notations and definitions from \cite{Cut} for the remainder of this article. We apply the techniques from \cite{Cut} to produce a filtration analogue of Corollary \ref{ineq3}, thereby generalising the third inequality given in Theorem $6.3.$ of \cite{Cut}.
\begin{theorem}\label{mainthm}
Suppose that $(R,m_R)$ is a Noetherian local ring of dimension $d\geq 1$ such that $\dim N(\hat{R})<d$, $M$ is a finitely generated $R$-module and $\mathcal{I}(1) = \{I(1)_j\}, \ldots, \mathcal{I}(r) = \{I(r)_j\}$ are filtrations of $R$ by $m_R$-primary ideals. Then $$e\left(\mathcal{I}(1)^{[d_1]},\ldots,\mathcal{I}(r)^{[d_r]};M\right)^d \leq \prod\limits_{i=1}^r e\left(\mathcal{I}(i);M\right)^{d_i}$$ for any $d_1,\ldots,d_r \in\mathbb{N}$ satisfying $d_1+\cdots+d_r=d$. 
\end{theorem}
\begin{proof}
In view of the reduction arguments applied in the beginning of the proof of Theorem $6.1.$ of \cite{Cut}, it suffices to prove the theorem for $R$ an analytically irreducible local domain and $M=R$. We shall first show that for all $a\in \mathbb{Z}_{\geq 1}$ the inequality holds for the $a$-th truncated filtrations $\mathcal{I}_a(1) = \{I_a(1)_t\}, \ldots, \mathcal{I}_a(r) = \{I_a(r)_t\}$ as defined in Definition $4.1.$ of \cite{Cut}. As the truncated filtrations are Noetherian, given $a\in \mathbb{Z}_{\geq 1}$, there exists $f_a\in \mathbb{Z}_{\geq 1}$ such that $I_a(i)_{f_at} = \left(I_a(i)_{f_a}\right)^t$ for all $t\geq 0$ and $1\leq i\leq r$ (see \cite{Bourbaki}[Proposition $3$, Section $1.3$, Chapter III]). Define filtrations of $R$ by $m_R$-primary ideals by $J_a(i)_t = I_a(i)_{f_at}$. From Theorem $17.4.2.$ of \cite{Hun}, we get that for $n_1,\ldots,n_r \in \mathbb{Z}_{\geq 1}$
\begin{multline*}
    \lim\limits_{t\to\infty}\dfrac{l_R\left(\faktor{R}{J_a(1)_{tn_1}\cdots J_a(r)_{tn_r}}\right)}{t^d} =\\ \sum_{\substack{d_1+\cdots + d_r = d\\ d_i\in\mathbb{N}}}\dfrac{1}{d_1!\ldots d_r!}e\left(J_a(1)_1^{[d_1]},\ldots, J_a(r)_1^{[d_r]};R\right)n_1^{d_1}\cdots n_r^{d_r}
\end{multline*}
\begin{align*}
    \lim\limits_{t\to\infty}\dfrac{l_R\left(\faktor{R}{J_a(i)_1^t}\right)}{t^d} &= \dfrac{1}{d!}e_R\left(J_a(i)_1;R\right) \quad (\text{for all $1\leq i\leq r$})
\end{align*}
where $e\left(J_a(1)_1^{[d_1]},\ldots, J_a(r)_1^{[d_r]};R\right)$ with $d_1+\cdots + d_r = d$ and $e_R\left(J_a(i)_1;R\right)$ for $1\leq i\leq r$, are the usual mixed multiplicities of ideals. Now by Corollary \ref{ineq3} we have that
\begin{align*}
    e\left(J_a(1)_1^{[d_1]},\ldots, J_a(r)_1^{[d_r]};R\right) &\leq \prod\limits_{i=1}^r e_R\left(J_a(i)_1;R\right)^{d_i}.
\end{align*}
From Theorem $6.1.$ of \cite{Cut}, we get that for $n_1,\ldots,n_r \in \mathbb{Z}_{\geq 1}$
\begin{multline*}
    \lim\limits_{t\to\infty}\dfrac{l_R\left(\faktor{R}{I_a(1)_{tn_1}\cdots I_a(r)_{tn_r}}\right)}{t^d} =\\ \sum_{\substack{d_1+\cdots + d_r = d\\ d_i\in\mathbb{N}}}\dfrac{1}{d_1!\ldots d_r!}e\left(\mathcal{I}_a(1)^{[d_1]},\ldots, \mathcal{I}_a(r)^{[d_r]};R\right)n_1^{d_1}\cdots n_r^{d_r}
\end{multline*}
\begin{align*}
    \lim\limits_{t\to\infty}\dfrac{l_R\left(\faktor{R}{I_a(i)_t}\right)}{t^d} &= \dfrac{1}{d!}e_R\left(\mathcal{I}_a(i);R\right) \quad (\text{for all $1\leq i\leq r$})
\end{align*}
where $e\left(\mathcal{I}_a(1)^{[d_1]},\ldots, \mathcal{I}_a(r)^{[d_r]};R\right)$ with $d_1+\cdots + d_r = d$ and $e_R\left(\mathcal{I}_a(i);R\right)$ for $1\leq i\leq r$, are the mixed multiplicities of filtrations. By Lemma $3.3.$ of \cite{Cut},
\begin{align*}
    \lim\limits_{t\to\infty}\dfrac{l_R\left(\faktor{R}{I_a(1)_{tn_1}\cdots I_a(r)_{tn_r}}\right)}{t^d} &= \dfrac{1}{f_a^d}\left(\lim\limits_{t\to\infty}\dfrac{l_R\left(\faktor{R}{J_a(1)_1^{tn_1}\cdots J_a(r)_1^{tn_r}}\right)}{t^d}\right)\\
    \lim\limits_{t\to\infty}\dfrac{l_R\left(\faktor{R}{I_a(i)_t}\right)}{t^d} &= \dfrac{1}{f_a^d} \lim\limits_{t\to\infty}\dfrac{l_R\left(\faktor{R}{J_a(i)_1^t}\right)}{t^d}
\end{align*}
for all $n_1,\ldots,n_r \in \mathbb{N}$ and $1\leq i\leq r$. Define the polynomial $H(x_1,\ldots,x_r)$ by 
\begin{multline*}
    H(x_1,\ldots,x_r) = \\ \sum_{\substack{d_1+\cdots + d_r = d\\ d_i\in\mathbb{N}}}\dfrac{1}{d_1!\ldots d_r!}\Big(e\left(\mathcal{I}_a(1)^{[d_1]},\ldots, \mathcal{I}_a(r)^{[d_r]};R\right)\\- \dfrac{1}{f_a^d}e\big(J_a(1)_1^{[d_1]},\ldots, J_a(r)_1^{[d_r]};R\big)\Big)x_1^{d_1}\cdots x_r^{d_r}.
\end{multline*}
Note that $H(x_1,\ldots,x_r)$ is a homogeneous polynomial in $x_1,\ldots,x_r$  with rational coefficients of total degree $d$ if $H\not\equiv 0$. It follows from the above identities that $H(n_1,\ldots,n_r)=0$ for all $n_1,\ldots,n_r \in\mathbb{Z}_{\geq 1}$. By using Lemma \ref{null} it follows that $H(x_1,\ldots,x_r)$ is identically zero. Hence
\begin{align*}
  e\left(\mathcal{I}_a(1)^{[d_1]},\ldots,\mathcal{I}_a(r)^{[d_r]};R\right) &= \dfrac{1}{f_a^d}e\left(J_a(1)_1^{[d_1]},\ldots, J_a(r)_1^{[d_r]};R\right)\\
  e\left(\mathcal{I}_a(i);R\right) &= \dfrac{1}{f_a^d} e\left(J_a(i)_1;R\right)
\end{align*}
 for all $d_1+\cdots+d_r = d$ and $1\leq i\leq r$. Therefore $$e\left(\mathcal{I}_a(1)^{[d_1]},\ldots, \mathcal{I}_a(r)^{[d_r]};R\right)^d \leq \prod\limits_{i=1}^r e_R\left(\mathcal{I}_a(i);R\right)^{d_i}$$ for all $a\in \mathbb{Z}_{\geq 1}$. Finally by Proposition $6.2.$ of \cite{Cut}, $$e\left(\mathcal{I}(1)^{[d_1]},\ldots, \mathcal{I}(r)^{[d_r]};R\right)^d \leq \prod\limits_{i=1}^r e_R\left(\mathcal{I}(i);R\right)^{d_i}.$$
\end{proof}
\section{Applications}
We obtain a generalisation of Corollary $1.7.$ of \cite{CSV} as a direct consequence of Theorem \ref{mainthm}.
\begin{corollary}
Suppose that $(R,m_R)$ is a Noetherian local ring of dimension $d\geq 1$ such that $\dim N(\hat{R})<d$, $M$ is a finitely generated $R$-module and $\mathcal{I}(1) = \{I(1)_j\}, \ldots, \mathcal{I}(r) = \{I(r)_j\}$ are filtrations of $R$ by $m_R$-primary ideals. Further assume that $e\left(\mathcal{I}(j);M\right) = 0$ for some $j$, $1\leq j\leq r$. Then $$e\left(\mathcal{I}(1)^{[d_1]},\ldots,\mathcal{I}(r)^{[d_r]};M\right) = 0$$ for any $d_1,\ldots,d_r \in\mathbb{N}$ satisfying $d_1+\cdots+d_r=d$ and $d_j>0$.
\end{corollary}
\begin{proof}
From Theorem \ref{mainthm}, we get that $$e\left(\mathcal{I}(1)^{[d_1]},\ldots,\mathcal{I}(r)^{[d_r]};M\right) \leq \prod\limits_{i=1}^r e\left(\mathcal{I}(i);M\right)^{\frac{d_i}{d}}$$ for any $d_1,\ldots,d_r \in\mathbb{N}$ satisfying $d_1+\cdots+d_r=d$. The conclusion now follows immediately from the above inequality.
\end{proof}
\bibliographystyle{plain}
\bibliography{main}

\begin{thebibliography}{1}

\bibitem{Alon}
N.~Alon and M.~Tarsi.
\newblock Colorings and orientations of graphs.
\newblock {\em Combinatorica}, 12:125--134, 1992.

\bibitem{Bourbaki}
N.~Bourbaki.
\newblock {\em Commutative Algebra, Chapters 1-7}.
\newblock Springer Verlag, 1989.

\bibitem{Cut}
Steven Cutkosky, Parangama Sarkar, and Hema Srinivasan.
\newblock Mixed multiplicities of filtrations.
\newblock {\em Transactions of the American Mathematical Society}, 372:1, 11
  2018.

\bibitem{CSV}
Steven Cutkosky, Hema Srinivasan, and Jugal Verma.
\newblock Positivity of mixed multiplicities of filtrations.
\newblock {\em Bulletin of the London Mathematical Society}, 52:335--348, 2020.

\bibitem{Hun}
Craig Huneke and Irena Swanson.
\newblock {\em Integral Closure of Ideals, Rings, and Modules}.
\newblock London Mathematical Society Lecture Note Series, 336. Cambridge
  University Press, Cambridge, 2006.

\bibitem{Laz}
R.~K. Lazarsfeld.
\newblock {\em Positivity in Algebraic Geometry I}.
\newblock Ergebnisse der Mathematik und ihrer Grenzgebiete, vols 48, Springer
  Verlag, Berlin, 2004.

\bibitem{BT}
Bernard Teissier.
\newblock Cycles \'{e}vanescents, sections planes et conditions de whitney,
  singularit\'es \`a carg\`ese.
\newblock {\em Ast\'{e}risque, Nos. 7 et 8. Soc. Math. France, Paris}, pages
  285--362, 1973.

\bibitem{TV}
N.V. Trung and J.K. Verma.
\newblock Hilbert functions of multigraded algebras, mixed multiplicities of
  ideals and their applications.
\newblock {\em Journal of Commutative Algebra}, 2(4):515--565, 2010.

\end{thebibliography}
\end{document}